\documentclass[10pt]{amsart}
\usepackage[dvipsnames,usenames]{color} 
\usepackage[toc,page]{appendix}

\usepackage{tikz}  
\usepackage{mathcomp,wasysym}  
\usepackage{mathrsfs}  
\usepackage{hyperref}
\usepackage{mathtools}
\usepackage[mathscr]{euscript}  
         
\usepackage{cite}      
\usepackage{graphicx}  
\usepackage[all]{xy} \xyoption{arc} \xyoption{color}

\usepackage{amsmath} 
\usepackage{amsthm} 
\usepackage{amsfonts}  
\usepackage{amssymb} 

\usepackage{verbatim}

\usepackage{mathrsfs}

\newcommand{\RR}{\mathbb R}

\newcommand{\bu}{v}

\newcommand{\ZZ}{{\mathbb Z}}

\newcommand{\Td}{{\mathbb T^d}}
\newcommand{\TT}{{\mathbb T}}

\newcommand{\pat}{\partial_t}
\newcommand{\pax}{\partial}

\newcommand{\vertiii}[1]{{\left\vert\kern-0.25ex\left\vert\kern-0.25ex\left\vert #1 
    \right\vert\kern-0.25ex\right\vert\kern-0.25ex\right\vert}}

\newenvironment{rcass}
  {\left.\begin{aligned}}
  {\end{aligned}\right\rbrace}

\newcounter{comentcount}
\newcounter{teocount}
\newtheorem{lem}{Lemma}

\newtheorem{corol}{Corollary}
\newtheorem{theorem}[teocount]{Theorem}  
\newtheorem{defi}{Definition}

\title[Fractional, logistic Keller-Segel]{Boundedness and homogeneous asymptotics for a fractional logistic Keller-Segel equations}

\author[J. Burczak]{Jan Burczak}
\email{jb@impan.pl,  jan.burczak@maths.ox.ac.uk}
\address{Institute of Mathematics, Polish Academy of Sciences, Warsaw, \'Sniadeckich 8, 00-656, Poland \\
OxPDE, Mathematical Institute, University of Oxford, UK}

\author[R. Granero-Belinch\'{o}n]{Rafael Granero-Belinch\'{o}n}
\email{granero@math.univ-lyon1.fr}
\address{Univ Lyon, Universit\'e Claude Bernard Lyon 1, CNRS UMR 5208, Institut Camille Jordan, 43 blvd. du 11 novembre 1918, F-69622 Villeurbanne cedex, France.}

\begin{document}

\begin{abstract}
In this paper we consider a $d$-dimensional ($d=1,2$) parabolic-elliptic Keller-Segel equation with a logistic forcing and a fractional diffusion of order $\alpha \in (0,2)$. We prove uniform in time boundedness of its solution in the supercritical range $\alpha>d\left(1-c\right)$, where $c$ is an explicit constant depending on parameters of our problem. Furthermore, we establish sufficient conditions for  $\|u(t)-u_\infty\|_{L^\infty}\rightarrow0$, where $u_\infty\equiv 1$ is the only nontrivial homogeneous solution. Finally, we provide a uniqueness result.
\end{abstract}

\maketitle 


\section{Introduction}

We consider the following drift-diffusion equation on $\TT^d=[-\pi,\pi]^d$  with periodic boundary conditions, $d=1,2$ (equivalently, on $\mathbb{S}^d$)
\begin{align}\label{eqDD}
\pat u&=- \Lambda^\alpha u+\chi\nabla\cdot(u \nabla v)+ru(1-u),& \text{ in }(x,t)\in \TT^d\times (0,\infty)\\
\Delta v -v&=u, &\text{ in }(x,t)\in \TT^d\times(0,\infty)\label{eqDD3}\\
u(x,0)&=u_0(x)\geq0 &\text{ in }x\in \TT^d,\label{eqDD2}
\end{align}
where $\Lambda^\alpha =(-\Delta)^{\alpha/2}$ with $0<\alpha< 2$. In this paper we will assume that $r<\chi$, which is the most difficult case from the perspective of our goal: studying the large time behaviour. This is due to the fact that then the potentially `destabilizing' term, whose influence is measured by $\chi>0$, is relatively powerful compared to the `homeostatic force' quantified by $r>0$.

Let us note that \eqref{eqDD}-\eqref{eqDD2} can be written as the following active scalar equation
$$
\pat u=- \Lambda^\alpha u+\chi\nabla\cdot(u B(u))+ru(1-u) \qquad \text{ in }(x,t)\in \TT^d\times(0,\infty),
$$
where the nonlocal operator $B$ is defined as
$$
B(u)=\nabla(\Delta-1)^{-1}u.
$$ 
In the remainder of this introduction, let us discuss some of the reasons for dealing with problem \eqref{eqDD} and the known results.
 
\subsection{Classical Patlak-Keller-Segel system}
Our interest in  \eqref{eqDD}-\eqref{eqDD2} follows from an aggregation equations related to the Patlak-Keller-Segel system. 
The classical (parabolic-elliptic) Patlak-Keller-Segel equation reads
\begin{equation}\label{cKS2}
\begin{aligned}
\pat u &= \Delta u + \chi \nabla\cdot(u \nabla v ), \\
\Delta v- \nu v&=  u.
\end{aligned}
\end{equation}
This system models \emph{chemotaxis}, \emph{i.e.} a chemically-induced motion of cells and certain simple organisms (e.g. bacteria, slime mold). In its more general version, it was proposed by Patlak \cite{patlak1953random} (in a different context of mathematical chemistry, hence his name is sometimes not used in the mathematical biology context) and Keller \& Segel \cite{keller1970initiation, keller1971model, keller1980assessing}, see also reviews by Blanchet \cite{blanchet2011parabolic} and Hillen \& Painter \cite{Hillen3}. In the biological interpretation, $u$ denotes density of cells (organisms) and $v$ stands for density of a chemoattractant. We will restrict ourselves to the (relevant biologically) case of $u \ge 0$, ensured by  $u_0 \ge 0$. The parameter $\chi>0$ quantifies the sensitivity of organisms to the attracting chemical signal and $\nu \ge 0$ models its decay\footnote{Observe that  the  Patlak-Keller-Segel equation is often written for unknowns $(u, -v)$}. 

Since $\nu>0$ plays a role of damping, let us for a moment consider \eqref{cKS2} on $\RR^2$ with $\nu=0$. Note that equation \eqref{cKS2} preserves the total mass ($\|u(0)\|_{L^1} = \|u(t)\|_{L^1}$). Furthermore, in the case $\nu=0$, the space $L^1$ is invariant under the scaling of the equation. It turns out that, despite its simplicity, the Patlak-Keller-Segel equation reveals in this setting an interesting global smoothness/blowup dichotomy. Namely, for $\|u(0)\|_{L^1} > {8 \pi}{\chi^{-1}}$ the classical solutions blow-up in $L^\infty$-norm in a finite time, for $\|u(0)\|_{L^1} < {8 \pi}{\chi^{-1}}$ they exist for all times (and are bounded), whereas for $\|u(0)\|_{L^1} = {8 \pi}{\chi^{-1}}$ they exist for all times but  their $L^\infty$-norm grow to infinity in time. The related literature is abundant, so let us only mention here the seminal results by J\"ager \& Luckhaus \cite{jager1992explosions} and  Nagai \cite{nagai1995blow}, the concise note by Dolbeault \& Perthame \cite{Dolbeault2}, where the threshold mass $8 \pi \chi^{-1}$ is easy traceable, as well as  Blanchet, Carrillo \& Masmoudi \cite{BCM}, focused precisely on the threshold mass case.

\subsection{Generalisations}
Our system \eqref{eqDD}-\eqref{eqDD2} differs from \eqref{cKS2} in two aspects: it involves the semilinearity $ru(1-u)$ and the fractional diffusion. We explain below what are both applicational and analytical reasons to consider each of these modifications separately. 
 
\subsubsection{Motivation for the logistic term}\label{ssec121}
Introduction of the logistic term $r u (1-u)$ in a biology-related equation is the (second) most classical way  to take into account a population dynamics (after the Malthusian exponential models, that do not cover the full lifespan of a population), compare formula (3) of Verhulst \cite{Verhulst} and model M8 of  \cite{Hillen3} in the context of chemotaxis. In agreement with the homeostatic character of the logistic function, the equation
\begin{equation}\label{cKS2l}
\begin{aligned}
\pat u &= \Delta u + \chi \nabla\cdot(u \nabla v ) +ru(1-u), \\
\Delta v- \nu v&=  u,
\end{aligned}
\end{equation}
is less prone to admit solutions that blow-up for $r>0$ than for $r=0$, compare Tello \& Winkler \cite{TelloWinkler}. What is interesting, blowups are in fact excluded for \emph{any initial mass}, no matter what is the relation between it and parameters $r, \chi$. For further results, including the parabolic-parabolic case, we refer to Winkler \cite{Winkler4, winkler2014global}.

Let us note that  a logistic term appears in the three-component urokinase plasminogen invasion model (see Hillen, Painter \& Winkler \cite{Hillen1}) and in a chemotaxis-haptotaxis model (see Tao \& Winkler \cite{TaoWinkler}). 

The question of the nonlinear stability of the homogeneous solution $u_\infty\equiv 1$, $v_\infty\equiv-1$ has received a lot of interest recently. For instance, Chaplain \& Tello \cite{chaplain2016stability} Galakhov, Salieva \& Tello \cite{galakhov2016parabolic} (see also Salako $\&$ Shen \cite{salako2017global}) studied the parabolic-elliptic Keller-Segel system and proved that if 
$r>2\chi$
then
$
\|u(t)-1\|_{L^\infty(\TT^d)}\rightarrow 0.
$
Let us note that the authors in \cite{chaplain2016stability, galakhov2016parabolic} did not provide with any explicit rate of convergence.

In the case of doubly parabolic Keller-Segel system, the question of stability of the homogeneous solution was addressed by Lin \& Mu \cite{lin2016global}, Winkler \cite{winkler2014global}, Xiang \cite{xiang2016strong} and Zheng \cite{zheng2017boundedness}, see also Tello \& Winkler \cite{TelloWinkler2}. For conditions forcing the solutions to vanish, compare  Lankeit \cite{Lankeit15}.

\subsubsection{Motivation for the fractional diffusion}
Since 1990's, a strong theoretical and empirical evidence has appeared for replacing the classical diffusion with a fractional one in Keller-Segel equations: $\Lambda^\alpha$, $\alpha<2$ instead of the standard $-\Delta=\Lambda^{2}$. Namely, in low-prey-density conditions, feeding strategies based on a L\'evy process (generated in its simplest isotropic-$\alpha$-stable version by $(-\Delta)^\frac{\alpha}{2} u$) are closer to optimal ones from theoretical viewpoint than strategies based on the Brownian motion (generated by $-\Delta u$). Furthermore, these strategies based on a L\'evy process are actually used by certain organisms. The interested reader can consult Lewandowsky, White \& Schuster \cite{Lew_nencki} for amoebas, Klafter,  Lewandowsky \& White \cite{Klaf90} as well as Bartumeus et al. \cite{Bart03} for microzooplancton, Shlesinger \& Klafter \cite{Shl86} for flying ants and Cole \cite{Cole} in the context of fruit flies. Surprisingly, even the feeding behavior of groups of large vertebrates is argued to follow L\'evy motions, the fact referred sometimes as to the \emph{L\'evy flight foraging hypothesis}. For instance, one can read Atkinson,  Rhodes, MacDonald \& Anderson \cite{Atk} for jackals, Viswanathan et al. \cite{Vnature} for albatrosses, Focardi, Marcellini \& Montanaro \cite{deers} for deers and Pontzer et al. \cite{hadza} for the Hadza tribe. 

Interestingly, the (fractional) Keller-Segel system can be recovered as limit cases of other equations. In this regards, Lattanzio \& Tzavaras \cite{lattanzio2016gas} considered the Keller-Segel system as high friction limits of the Euler-Poisson system with attractive potentials (note that the case with fractional diffusion corresponds to the nonlocal pressure law $p(u)=\Lambda^{\alpha-2}u(x)$) while Bellouquid, Nieto \& Urrutia \cite{bellouquid2016kinetic} obtained the fractional Keller-Segel system as a hydrodynamic limit of a kinetic equation (see also Chalub, Markowich, Perthame \& Schmeiser \cite{chalub2004kinetic}, Mellet, Mischler \& Mouhot\cite{mellet2011fractional}, Aceves-Sanchez \& Mellet \cite{aceves2016asymptotic} and Aceves-Sanchez \& Cesbron \cite{aceves2016fractional}). 

\vskip 1mm
In view of the last two paragraphs, our aim to onsider the combined effect of (regularizing) logistic term and (weaker than classical) fractional diffusion is both analytically interesting and reasonable from the viewpoint of applications.

\subsection{Prior results for the Keller-Segel systems with fractional diffusions}\label{ssec1.5} Let us  recall now certain analytical results for the fraction Keller-Segel systems and its generalisations.

The system \eqref{cKS2} is part of a larger family of aggregation-diffusion-reaction systems
\begin{equation}\label{eq:1}
\left\{\begin{aligned}
\pat u&=-\Lambda^\alpha u-\chi\nabla\cdot (u K(v))+F(u),\\
\tau\pat v&=\kappa\Delta v+G(u,v),
\end{aligned}\right.
\end{equation}
$\alpha \in (0,2)$.
The system \eqref{eq:1} is referred to as a `parabolic-parabolic' one if $\tau,\kappa>0$, `parabolic-elliptic' if $\tau=0$, $\kappa>0$ and `parabolic-hyperbolic' if $\tau>0$, $\kappa=0$. 
For a more exhaustive discussion of these models, we refer to the extensive surveys by Hillen \& Painter \cite{Hillen3}, Bellomo, Bellouquid, Tao \& Winkler \cite{bellomo2015towards} and Blanchet \cite{blanchet2011parabolic}. 

In what follows, let us recall known results, in principle for the following (generic) choices $F(u) = r u (1-u)$, $r \ge 0$ and $G(u,v)=u-v$ or $G(u,v)=u$. The first interaction operator $K$ that one should have in mind is the most classical $K(v) = \nabla v$, but other choices are studied, that critically influence the system's behavior.

\subsubsection{Case of no logistic term $r=0$}

Since  $ -\Lambda^\alpha u$ provides for $\alpha<2$ a weaker dissipation than the classical one, it is expected that a blowup may occur. This is indeed the case for the generic fractional parabolic-elliptic cases in $d\geq2$, compare for instance results by Biler, Cie{\'s}lak, Karch \& Zienkiewicz  \cite{biler2014local} and  Biler \& Karch \cite{biler2010blowup}. The results for other interaction operators can be found in a vast literature on aggregation equations, not necessarily motivated by mathematical biology, including Biler, Karch \& Lauren\c{c}ot  \cite{biler2009blowup},   Li \& Rodrigo \cite{li2009finite, li2010exploding, li2009refined, li2010wellposedness}. Naturally, there are small-data global regularity results available, compare e.g. Biler \& Wu \cite{BilerWu} or \cite{BG2}. To the best of our knowledge, the question of global existence vs. finite time blow up of the fully parabolic Keller-Segel system ($\tau,\kappa>0$) with fractional diffusion and arbitrary initial data remains open (compare with Wu \& Zheng \cite{WuZheng} and \cite{BG2}). Similarly, as far as we know, the finite time blow up for the parabolic-hyperbolic case (the extreme case $\kappa=0$), remains an open problem even for low values of $\alpha$, compare \cite{Ghyperparweak,Ghyperparstrong}.

The $1$d case received much attention in the recent years. Let us review here some of the related results 
\begin{itemize}

\item A majority of the currently available results concerns the parabolic-elliptic Keller-Segel system ($\tau=0$, $\kappa\neq0$). In this context it is natural to look for a minimal strength of diffusion that gives rise to global in time smooth solutions. Escudero \cite{escudero2006fractional} proved that $\alpha>1$ leads to global existence of solutions in the large (i.e. without data smallness). Next, Bournaveas \& Calvez \cite{bournaveas2010one} obtained finite time blow up in the supercritical case $\alpha<1$ and established that for $\alpha=1$ there exists a (non-explicit) constant $K$ such that $\|u_0\|_{L^1}\leq K$ implies global in time solutions. Such a constant was later explicitly estimated as ${2\pi}^{-1}$ in Ascasibar, Granero-Belinch\'on and Moreno \cite{AGM} and improved in \cite{BG}. It was also conjectured that the case $\alpha=1$ is critical, i.e. that a large $\|u_0\|_{L^1}$ leads to a finite-time blowup, see \cite{bournaveas2010one}. Quite recently we were able to disprove that conjecture in \cite{BG3} by showing that, regardless of the size of initial data, the smooth solution exists for arbitrary large times (but our global bound is unfortunately not uniform in time yet). 

\item The parabolic-parabolic problem was considered in \cite{BG2}, both without the logistic term and with it. In the former case, beyond a typical short-time existence result and continuation criteria, we showed smoothness and regularity for $\alpha>1$ as well as, under data smallness, for  $\alpha=1$. Further results for the logistic case will be recalled in the next section.

\item The parabolic-hyperbolic problem $\tau>0$, $\kappa=0$, was proposed by Othmer \& Stevens \cite{stevens1997aggregation} as a model of the movement of myxobacteria. However, this model has been used also to study the formation of new blood vessels from pre-existing blood vessels (see  Corrias, Perthame \& Zaag \cite{corrias2003chemotaxis}, Fontelos, Friedman \& Hu \cite{fontelos2002mathematical}, Levine, Sleeman, Brian \& Nilsen-Hamilton \cite{levine2000mathematical}, Sleeman, Ward \& Wei \cite{sleeman2005existence}). Because of this, this system captured the interest of numerous researchers (see \cite{corrias2003chemotaxis, Ghyperparweak,Ghyperparstrong, fontelos2002mathematical, zhang2013global, li2010nonlinear,xie2013global, zhang2015global, fan2012blow, li2015initial, mei2015asymptotic,li2015quantitative,li2009nonlinear,zhang2007global,li2014stability,wang2008shock,
wang2016asymptotic,li2011hyperbolic,hao2012global,li2012global} and the references therein).
\end{itemize}

\subsubsection{Case with logistic term $r>0$}

Let us first quickly recall our  $1$d results in \cite{BG2} for parabolic-parabolic fractional Keller-Segel with logistic term, beyond those holding without it. We obtained global-in-time smoothness for $\alpha \ge 1$.  Interestingly, partially due to the logistic term, the considered system shows spatio-temporal chaotic behavior with peaks that emerge and eventually merge with other peaks. In that regard, we studied the qualitative properties of the attractor and obtained bounds for the number of {peaks}. This number may be related to dimension of the attractor. Mathematically, this estimate was obtained with a technique applicable to other problems with chaotic behavior, compare for instance \cite{GH}.

The currently available regularity results are much better for the parabolic-elliptic case. It turns out that the logistic term provides enough stabilisation to allow for global-in-time smooth solutions even for certain `supercritical' regime of diffusions $\alpha<d$. Namely, we have considered the $1$d case in \cite{BG4} and $2$d case in \cite{burczak2016suppression}. Let us recall some of these results, focusing on the potentially most singular case $r< \chi$, since it is within the scope of this note. For any
\[
\alpha > d\left(1-\frac{r}{\chi}\right)
\]
the problem \eqref{eqDD}-\eqref{eqDD2} enjoys global in time smooth solutions, but with no uniform-in-time bounds (i.e. without excluding the infinite-time-blowup). More precisely, we obtained in \cite{BG4, burczak2016suppression} that 
\begin{align}
\max_{0\leq t \leq T}\|u\|_{L^{\frac{\chi}{\chi-r}}}&\leq e^{rT}\|u_0\|_{L^{\frac{\chi}{\chi-r}}}\label{bound2}\\ 
\max_{0\leq t \leq T}\|u\|_{L^{p}}&\leq c_1( e^{c_2 T} +1)\|u_0\|^{c_2}_{L^{p}} \quad \text{for any finite $p$} \label{bound2s}\\ 
\max_{0\leq t \leq T}\|u\|_{L^{\infty}}&\leq c_2  e^{c_1 T},\label{bound3}
\end{align} 
where $c_1 (\| u_0\|_1, p, r,\chi, \alpha,d)$ and $c_2 (p, r,\chi, \alpha,d)$.
For $d=2$, the estimate \eqref{bound2} is given as (4.11) of Lemma 4.3 in \cite{burczak2016suppression}, \eqref{bound2s} occupies Lemma 4.4 there and \eqref{bound3} follows from computations leading to the estimate of Theorem 2. For $d=1$ the analogous results come from \cite{BG4} (some of them are not stated explicitly there, but they follow the lines of the $2$d case).

Up to now, the only uniform in time bounds we were able to provide concerned the $1$d case and they were far from satisfactory ones. They involved either dissipations that  clearly outweigh aggregation ($\alpha > d=1$) or certain smallness assumptions. For instance, for $\alpha=1, d=1$ 
\begin{equation}\label{eq:advS}
\chi<r+\frac{1}{2\pi\max\{\|u_0\|_{L^1},2\pi\}}
\end{equation}
implies an uniform in time bound
\begin{equation}\label{bound61}
\max_{0\leq t \leq T}\|u\|_{L^{\infty}}\leq c_3(r,\chi,u_0),
\end{equation}
see \cite{BG4}, Proposition 1.

\subsection{Purpose of this note}
In the case $r=0$, the system \eqref{eqDD}-\eqref{eqDD2} with $\alpha<d$ develops finite-time blowups. Hence the regime 
\[
d> \alpha > d\left(1-\frac{r}{\chi}\right),
\]
where our just-recalled existence result on global-in-time smooth solutions holds, can be seen as an interestingly 'supercritical' one. However, the non-uniformity in time of our global bounds \eqref{bound2}-\eqref{bound3} appeared to us far from optimal. 

Consequently, in this note,  we sharpen the estimates \eqref{bound2}-\eqref{bound3} to time-independent ones. Moreover, we provide conditions that ensure the convergence of the solution $u$ towards the only nontrivial homogeneous steady state $u_\infty\equiv1$, including some speed of convergence estimates. We present also a `semi-strong' uniqueness result. For statements of our results, we refer to Section \ref{sec:mr}.

\subsection{Notation for functional spaces}
Let us write $\pax^n,$ $n\in\mathbb{Z}^+$, for a generic derivative of order $n$. Then, the fractional $L^p$-based Sobolev spaces $W^{s,p}(\Td)$ (also known as Sobolev-Slobodeckii or Besov spaces $B^{s,p}_p(\Td)$) are 
$$
W^{s,p} (\Td)=\left\{f\in L^p(\Td) \; | \quad \pax^{\lfloor s\rfloor} f\in L^p(\Td), \frac{|\pax^{\lfloor s\rfloor}f(x)-\pax^{\lfloor s\rfloor}f(y)|}{|x-y|^{\frac{d}{p}+(s-\lfloor s\rfloor)}}\in L^p(\Td\times\Td)\right\},
$$
endowed with the norm
$$
\|f\|_{W^{s,p}}^p=\|f\|_{L^p}^p+\|f\|_{\dot{W}^{s,p}}^p, 
$$
$$
\|f\|_{\dot{W}^{s,p}}^p=\|\pax^{\lfloor s\rfloor} f\|^p_{L^p}+\int_{\Td}\int_{\Td}\frac{|\pax^{\lfloor s\rfloor}f(x)-\pax^{\lfloor s\rfloor}f(y)|^p}{|x-y|^{d+(s-\lfloor s \rfloor)p}}dxdy.
$$
In the case $p=2$, we write $H^s(\Td)=W^{s,2}(\Td)$ for the standard non-homogeneous Sobolev space with its norm
$$
\|f\|_{H^s}^2=\|f\|_{L^2}^2+\|f\|_{\dot{H}^s}^2, \quad \|f\|_{\dot{H}^s}=\|\Lambda^s f\|_{L^2}.
$$ 
Next, for $s\in (0,1)$, let us denote the usual H\"older spaces as follows
$$
C^{s} (\Td)=\left\{f\in C(\Td) \;| \quad \frac{|f(x)-f(y)|}{|x-y|^{s}}\in L^\infty(\Td\times\Td)\right\},
$$
with the norm
$$
\|f\|_{C^{s}}=\|f\|_{L^\infty}+\|f\|_{\dot{C}^{s}},\quad
\|f\|_{\dot{C}^{s}}=\sup_{(x,y)\in\Td\times\Td}\frac{|f(x)-f(y)|}{|x-y|^{s}}.
$$
For brevity, the domain dependance of a function space will be generally suppressed. Finally, we will use the standard notation for evolutionary (Bochner) spaces, writing $L^p(0,T; W^{s,p})$ etc. In case of suppressing the time and space domain, the outside is always time related, i.e. $L^p(L^q)$ denotes $L^p(0,T; L^q (\TT^d) )$.

\section{Main results and their discussion}\label{sec:mr}

\subsection{Classical solvability and uniform-in-time boundedness}
 In our first result, we prove that $\|u(t)\|_{L^\infty(\TT^d)}$ remains in fact uniformly bounded. In order to compute the bound, let us introduce the following numbers
\[
\mathscr{C}_{d,\alpha}=2\left(\int_{\RR^d}\frac{4\sin^2\left(\frac{x_1}{2}\right)}{|x|^{d+\alpha}} dx\right)^{-1}, \qquad \mathscr{P}_{d,\alpha}= \frac{2 \mathscr{C}_{d,\alpha}}{\left(2\pi \right)^{\alpha}d^{\frac{d+\alpha}{2}}},
\]
and for any $\epsilon \in (0, r)$, $p=\frac{\chi}{\chi-r+\epsilon}$
\[
\mathscr{M}_1(d,p,\alpha)=\left(\frac{\pi^{d/2}}{2^{1+p}}\int_{0}^\infty z^{d/2}e^{-z}dz\right)^{1/p}, \qquad \mathscr{M}_2(d,p,\alpha)=\mathscr{C}_{d,\alpha}\frac{\left(\frac{\pi^{d/2}}{\int_{0}^\infty z^{d/2}e^{-z}dz}\right)^{1+\alpha/d}}{4\cdot 2^{\frac{(p+1)\alpha}{d}}}
\]
and quantities
\[
\mathcal{R}_0(r,\epsilon,\chi,d,\alpha, u_0)= \left(\frac{ r}{\mathscr{P}_{d,\alpha}} \left(\frac{r}{\epsilon}\frac{\chi}{2\chi-r+\epsilon}\right)^{\frac{\chi}{\chi-r+\epsilon}}+\max \left\{ (2\pi)^{-d} \|u_0\|^2_{L^1(\TT^d)},(2\pi)^d\right\} \right)^
{1-\frac{r-\epsilon}{\chi}}
\]
$$
\mathcal{\tilde R}_2(r,\epsilon,\chi,d,\alpha)=  \left(\frac{ r}{\mathscr{P}_{d,\alpha}} \left(\frac{r}{\epsilon}\frac{\chi}{2\chi-r+\epsilon}\right)^{\frac{\chi}{\chi-r+\epsilon}}+ 3 (2\pi)^d \right)^
{1-\frac{r-\epsilon}{\chi}},
$$
with the latter being a data-independent one.
They are needed for (uniformly bounded in time)
\[
\begin{aligned}
\mathcal{Q}_0 (t; r,\epsilon,\chi,d,\alpha, u_0) =&  \|u_0\|_{L^{\frac{\chi}{\chi-r+\epsilon}}} e^{- \mathscr{P}(d,\alpha) t } + (1- e^{- \mathscr{P}(d,\alpha) t })  \mathcal{R}_0 , \\
 \mathcal{\tilde Q}_2  (t; r,\epsilon,\chi,d,\alpha, u_0)  =& \|u({t_0= r^{-1} \ln2})\|_{L^{\frac{\chi}{\chi-r+\epsilon}}} e^{- \mathscr{P}(d,\alpha) t } + (1- e^{- \mathscr{P}(d,\alpha) t })\mathcal{\tilde R}_2  
\end{aligned}
\]
that are involved in
\[
 \mathscr{ R}_3 (t; r,\epsilon,\chi,d,\alpha, u_0) = 2 e^{-t} \|u_0\|_{L^\infty(\TT^d)} +2 \mathcal{Q}^{\frac{3}{\sigma}}_0 (\mathscr{M}_1+ {\left(\frac{4 \chi}{\mathscr{M}_2}\right)}^{ \frac{1}{2} +\frac{1}{\sigma}} +1 )
\]
and
\[
 \mathscr{ \tilde R}_3   (t; r,\epsilon,\chi,d,\alpha, u_0)  = 2 e^{-t}  \|u({t_0= r^{-1} \ln 2}) \|_{L^\infty(\TT^d)} +2 \mathcal{\tilde Q}^{\frac{3}{\sigma}}_2 (\mathscr{M}_1+ {\left(\frac{4 \chi}{\mathscr{M}_2}\right)}^{ \frac{1}{2} +\frac{1}{\sigma}} +1 )
\]
again uniformly bounded in time.
Observe that 
\begin{equation}\label{Rinf}
\mathscr{ \tilde R}_\infty   (r,\epsilon,\chi,d,\alpha) =  \lim_{t \to \infty} \mathscr{ \tilde R}_3   (t; r,\epsilon,\chi,d,\alpha, u_0)  = 2 \mathcal{\tilde R}^{\frac{3}{\sigma}}_2 (\mathscr{M}_1+ {\left(\frac{4 \chi}{\mathscr{M}_2}\right)}^{ \frac{1}{2} +\frac{1}{\sigma}} +1 )
\end{equation}
is additionally $u_0$-independent. Having the above notions, we are ready to state

\begin{theorem}\label{thm1}
Let $u_0\in H^{d+2}$ be nonnegative, $\alpha \in (0,2)$ and  $\chi>r>0$. Then, as long as
$$
\alpha>d\left(1-\frac{r}{\chi}\right),
$$
 the problem \eqref{eqDD}-\eqref{eqDD2} admits a nonnegative classical solution
\[
u \in C (0, T; H^{d+2} (\TT^d)) \quad \cap \quad C^{2, 1} (\TT^d \times (0,T))
\]
for any finite $T$ (with nonpositive $v$ solving \eqref{eqDD3}).

Moreover, $u$ is uniformly bounded in time: 
\begin{equation}\label{boundn1}
\begin{aligned}
\|u(t)\|_{L^\infty(\TT^d)} &\leq 
 \mathscr{ R}_3 (t; r,\epsilon,\chi,d,\alpha, u_0) \quad \forall_{t \ge 0},  \\
 \|u(t)\|_{L^\infty(\TT^d)} &\leq \mathscr{ \tilde R}_3 (t; r,\epsilon,\chi,d,\alpha, u_0) \quad \forall_{t \ge r^{-1} \ln 2 },
 \end{aligned}
\end{equation}
hence in particular
\begin{equation}\label{boundninf}
\limsup_{t \to \infty}  \|u(t)\|_{L^\infty(\TT^d)} \le \mathscr{ \tilde R}_\infty   (r,\epsilon,\chi,d,\alpha) 
\end{equation}
no matter what the initial data are.
\end{theorem}
Let us recall that the condition $\chi>r$ was chosen not as a simplification, but to focus ideas. In fact, it is the most demanding case. It can be easily seen, since writing $\bar u (t) = max_x u (x, t) = u (x_t, t) $, $v_u (t) =  v (x_t, t)$ we have the ODI (formally, but easily made rigorous)
\[
\frac{d}{dt} \bar u \le 0 + \chi (0 + \bar u \Delta v_u) + r \bar u  - r \bar u^2 \le  \chi  \bar u (\bar u + v_u) + r \bar u  - r \bar u^2 \le   r \bar u  - (r- \chi) \bar u^2
\]
which is a Bernoulli ODI, so 
\[
 \bar u (T) \le \frac{ \bar u (0) }{e^{-rT} + \frac{r- \chi}{r} (1-e^{-rT}) \bar u (0)} \to r/ (r- \chi) \quad \text{ as } T \to \infty.
\]

\subsection{Stability of the homogeneous solution $u_\infty=1$, $v_\infty=-1$}
The first result here ensures the exponential convergence $u\rightarrow u_\infty$, as long as $\chi$ and $r$ are close to each other in terms of the initial data.
Let us define the time-independent upper bound for $\mathscr{ R}_3 (t; r,\epsilon,\chi,d,\alpha, u_0)$ of \eqref{boundn1} via
\[
\mathscr{ \bar R}_3 (r,\epsilon,\chi,d,\alpha, u_0) = 2  \|u_0\|_{L^\infty(\TT^d)} +2 \mathcal{Q}^{\frac{3}{\sigma}}_0 (\mathscr{M}_1+ {\left(\frac{4 \chi}{\mathscr{M}_2}\right)}^{ \frac{1}{2} +\frac{1}{\sigma}} +1 )
\]

\begin{theorem}[Stability of the homogeneous solution I]\label{thm2}
Let $u\in H^{d+2}$ be the classical solution to \eqref{eqDD}-\eqref{eqDD2} starting from $u_0\in H^{d+2}$, $u_0 \not\equiv 0$, $u_0 \ge 0$. Assume that $\alpha \in (0,2)$ and  $\chi>r>0$. Let $\alpha$ be such that 
$$
\alpha>d\left(1-\frac{r}{\chi}\right).
$$
Moreover, assume that $\chi>r$ are close enough in terms of data, so that
\begin{equation}\label{ssc1}
-\gamma:=2\chi-r+2(\chi-r)\left( \mathscr{ \bar R}_3 (r,\epsilon,\chi,d,\alpha, u_0) -1\right)-\frac{(2\pi)^d\mathscr{C}_{d,\alpha}}{({2}\pi\sqrt{d})^{d+\alpha}} < 0.
\end{equation}
Then 
$$
\|u(t)-1\|_{L^\infty(\TT^d)}\leq
\left(\|u(t)\|_{L^\infty(\TT^d)}-\min_{x\in\TT^d}u(x,t)\right)\leq \left(\|u_0\|_{L^\infty(\TT^d)}-\min_{x\in\TT^d}u_0(x)\right)e^{-\gamma t} 
$$
\end{theorem}
\begin{corol}\label{cor:1}
In condition \eqref{ssc1} we can replace $\mathscr{\bar R}_3$ with the initial-data independent $\mathscr{ \tilde R}_\infty$ of  \eqref{boundninf}, at the cost of having the statement valid only for times $t^* \ge t (r,\epsilon,\chi,d,\alpha, u_0)$ 
\end{corol}

Our second stability result for  $u_\infty=1$, $v_\infty=-1$ concerns the `critical' case $\alpha =1$ in $d=1$. We do not obtain rate of convergence as before, but the conditions on $\chi$ and $r$ are straightforward.
\begin{theorem}[Stability of the homogeneous solution II]\label{thm2b}
Assume that $d=\alpha=1$. Let $u\in H^{3}$ be the classical solution to \eqref{eqDD}-\eqref{eqDD2} starting from $u_0\in H^{3}$,  $u_0 \not\equiv 0$, $u_0 \ge 0$. Assume that $\alpha \in (0,2)$, $\chi > r >0$ and 
$$
\chi<\frac{1}{8\pi^2}.
$$
Then we have that
$$
\|u(t)-1\|_{L^\infty(\TT^d)}\rightarrow 0.
$$
\end{theorem}
Let us recall from Section \ref{ssec121} that question of the nonlinear stability of the homogeneous solution $u_\infty\equiv 1$, $v_\infty\equiv-1$ in classical (more diffusive) setting has received a lot of interest recently \cite{chaplain2016stability,galakhov2016parabolic,salako2017global}.

\subsection{Uniqueness}
The previous theorems employ  notions of solutions with high regularity. However, for the sake of our uniqueness result (Theorem \ref{thm3} below), let us introduce
\begin{defi}\label{def:1}
If a  function 
 \[u \in L^2(0,T;L^2(\TT^d))
 \] satisfies \eqref{eqDD}, \eqref{eqDD2} in the following sense
\begin{align*}
{\int_{\TT^ d} u_0 \varphi(0)dx} -\int_0^T \int_{\TT^ d} u \pat \varphi dx ds +\int_0^T \int_{\TT^ d}  u  \Lambda^{\alpha}  \varphi dxds &= -\chi \int_0^T \int_{\TT^ d}  (u B(u))  \nabla \varphi dxds \\
&\quad+ r \int_0^T \int_{\TT^ d} u(1-u)  \varphi dxds
\end{align*}
for a sufficiently smooth $\varphi$, it is called a \emph{distributional solution} to \eqref{eqDD}, \eqref{eqDD2}.
\end{defi}

It holds
\begin{theorem}\label{thm3}
Let $\alpha>1$ and  $r \ge 0$. Nonnegative solutions $u\in L^\infty (L^2) \cap L^2 (H^\alpha)$ to \eqref{eqDD}-\eqref{eqDD2} are unique (in $L^\infty (L^2) \cap L^2 (H^\alpha)$ class).
\end{theorem}

\subsection{Discussion}
In Theorem \ref{thm1}, we prove the uniform-in-time boundedness of the solution to \eqref{eqDD}-\eqref{eqDD2} (regardless of the size of the initial data) when the dissipation strength lies in the regime
$$
\alpha>d\left(1-\frac{r}{\chi}\right). 
$$
In particular, the estimate \eqref{boundn1} in Theorem \ref{thm1} sharpens both \eqref{bound3} by excluding any (in particular, exponential) dependence on $T$ and \eqref{bound61} by removing the additional assumptions \eqref{eq:advS}. 
Let us also mention that we provide our results via a new and much shorter reasoning than that of \cite{BG4, burczak2016suppression}.

In Theorem \ref{thm2} and its Corollary \ref{cor:1} we prove some conditions (one of them depends on lower norms of $u_0$) that lead to the nonlinear stability of the homogeneous solution $u_\infty\equiv1$. Furthermore, Theorem \ref{thm2} also proves that the decay towards the equilibrium state $u_\infty$ is exponential with a explicitly computable rate. 

In Theorem \ref{thm2b}, we show for the case  $d=\alpha=1$ that $\chi< (8 \pi^2)^{-1}$ suffices for  $u_\infty \equiv 1$ to be asymptotically stable. Let us emphasize it is a phenomenon {independent of} $u_0$.

Let us compare our results with the previous ones in \cite{chaplain2016stability, galakhov2016parabolic, salako2017global}
\begin{itemize}
\item Both Theorems \ref{thm2} and \ref{thm2b} cover the case of (weaker) fractional dissipations $0<\alpha<2$, while the previous ones hold for the classical laplacian $\alpha=2$, but, on the other hand, some of the previous results are valid for an arbitrary space dimension.
\item Both our Theorems \ref{thm2} and \ref{thm2b} consider the case $r<\chi$, while the previous impose at best $2\chi<r$.
\item Theorem \ref{thm2} provides an exponential decay with a computable rate.
\end{itemize}

As far as we know, the available uniqueness results for Keller-Segel-type systems are standard ones, regarding uniqueness of classical solutions. Theorem \ref{thm3} indicates in particular that classical solutions are unique within any  $L^\infty (L^2) \cap L^2 (H^\alpha)$ solution, as long as $\alpha>1$. On one hand, it is a relaxation of standard classical uniqueness, but on the other hand it would be more natural to look for a weak-strong uniqueness result, with `weak' part related to certain simple global energy estimates. Since however actually these are no more than $L^1$, $LlogL$ ones (or slightly better with the logistic term, compare Lemma 3 of \cite{BG4}), a satisfactory weak-strong result remains open for now.

\section{Proof of Theorem \ref{thm1} (classical solvability and uniform-in-time bounds)}
The stated regularity
\[
u \in C (0, T; H^{d+2} (\TT^d)) \quad \cap \quad C^{2, 1} (\TT^d \times (0,T))
\]
for any finite $T$, follows from the main results of \cite{BG4, burczak2016suppression}. For nonnegative data $u_0$, the corresponding solution $u(t)$ is also nonnegative. Furthermore the solution to \eqref{eqDD3} satisfies 
\begin{equation}\label{eq:signv}
v \le 0.
\end{equation}
To realise that, it suffices to consider $x_t$ such that 
$$
v (t, x_t) = \max_{y} v (t, y),
$$
then we have
\[
0 \le u(x_t) = \Delta v(x_t) - v(x_t) \implies v(x_t) \le \Delta v(x_t) \le 0.
\]
Furthermore it holds
\begin{equation}\label{propv}
-\min_x u(x,t)\geq v(x,t)\geq -\max_x u(x,t).
\end{equation}
Therefore, to conclude Theorem \ref{thm1} it suffices to prove the uniform estimate \eqref{boundn1}. Define
\begin{equation}\label{R1}
\mathcal{R}_1(u_0,d)=\max\left\{(\|u_0\|_{L^1(\TT^d)},(2\pi)^d\right\}
\end{equation}
Then the solution $u$ verifies the following estimates
\begin{align}
\sup_{0\leq t<\infty}\|u(t)\|_{L^1(\TT^d)}&\leq \mathcal{R}_1(u_0,d), \label{LinfL1}\\
\limsup_{t\rightarrow\infty}\|u(t)\|_{L^1(\TT^d)}&\leq (2\pi)^d, \label{LinfL12}\\
\int_0^T\|u(s)\|_{L^2(\TT^d)}^2ds&\leq \frac{\|u_0\|_{L^1(\TT^d)}}{r}+\mathcal{R}_1(u_0,d)T, \label{L2L2}
\end{align}
The estimate  \eqref{L2L2} follows for  $d=2$ from \cite{burczak2016suppression}, Lemma 4.3 and for  $d=1$ from \cite{BG4}, Lemma 4. To justify the estimates \eqref{LinfL1},  \eqref{LinfL12}, we use for $\eta (t) = \|u(t)\|_{L^1(\TT^d)}$ the (Bernoulli) ODI
\begin{equation}\label{odi}
\frac{d}{dt} \eta (t)  \leq r \eta (t) - r(2\pi)^{-d} \eta^2 (t)
\end{equation}
following from integration in space of \eqref{eqDD} and the Jensen inequality. Introducing $\kappa$ through $1= \kappa (\delta + \eta)$ we obtain after $\delta \to 0$ that
\[
\eta (T) \le \frac{\eta (0) }{e^{-rT} + \frac{r(2\pi)^{-d}}{r} (1-e^{-rT}) \eta (0)}.
\]
Hence
\begin{equation}
\|u(T)\|_{L^1(\TT^d)} \le \frac{  \|u_0\|_{L^1(\TT^d)}}{e^{-rT} + (2\pi)^{-d} (1-e^{-rT})  \|u_0\|_{L^1(\TT^d)}}.  \label{LinfL3}
\end{equation}
Considering in  \eqref{LinfL3} large $T$ implies \eqref{LinfL12}, while the uniform bound for r.h.s. of  \eqref{LinfL3} implies \eqref{LinfL1}.


We will find useful the next lemma stating the uniform-in-time boundedness in some $L^p$ norm for $p$ very close to 1. For its formulation, let us recall that $\mathscr{P}(d,\alpha)$ comes  from Lemma \ref{lemapoincare}
and let us define
\begin{equation}\label{R2}
\mathcal{R}_2 = \mathcal{R}_2(r,\epsilon,\chi,d,\alpha, u_0)= \left(\frac{ r\left(\frac{r}{\epsilon}\frac{\chi}{2\chi-r+\epsilon}\right)^{\frac{\chi}{\chi-r+\epsilon}}}{\mathscr{P}(d,\alpha)}+\frac{\mathcal{R}_1(u_0,d)^2}{(2\pi)^d}+\mathcal{R}_1(u_0,d)\right)^
{1-\frac{r-\epsilon}{\chi}},
\end{equation}
with $\mathcal{R}_1(u_0,d)$ defined in \eqref{R1}, as well as recall that $\mathcal{\tilde R}_2$ is defined as
\begin{equation}\label{R2t}
\mathcal{\tilde R}_2 = \mathcal{\tilde R}_2(r,\epsilon,\chi,d,\alpha)=  \left(\frac{ r\left(\frac{r}{\epsilon}\frac{\chi}{2\chi-r+\epsilon}\right)^{\frac{\chi}{\chi-r+\epsilon}}}{\mathscr{P}(d,\alpha)}+ 3 (2\pi)^d \right)^
{1-\frac{r-\epsilon}{\chi}}.
\end{equation}

Observe the last quantity is data-independent.
\begin{lem}\label{keylemma}
Let $d=1$ or $2$ and
\[
u \in C (0, T; H^{d+2} (\TT^d)) \quad \cap \quad C^{2, 1} (\TT^d \times (0,T))
\] solve \eqref{eqDD}-\eqref{eqDD2} starting from nonnegative $u_0\in H^{d+2}$. Assume that $\chi>r$, fix any $0<\epsilon<r$. Then
\begin{equation}\label{eqLp}
\max_{0\leq t <\infty}\|u\|_{L^{\frac{\chi}{\chi-r+\epsilon}}(\TT^d)}\leq  \|u_0\|_{L^{\frac{\chi}{\chi-r+\epsilon}}} e^{- \mathscr{P}(d,\alpha) t } + (1- e^{- \mathscr{P}(d,\alpha) t })  \mathcal{R}_2(r,\epsilon,\chi,d,\alpha, u_0).
\end{equation}
Furthermore, we have that
\begin{equation}\label{limsupLpa}
\max_{r^{-1} \ln2 \leq t <\infty} \|u(t)\|_{L^{\frac{\chi}{\chi-r+\epsilon}}} \le \|u({t_0= r^{-1} \ln2})\|_{L^{\frac{\chi}{\chi-r+\epsilon}}} e^{- \mathscr{P}(d,\alpha) t } + (1- e^{- \mathscr{P}(d,\alpha) t })\mathcal{\tilde R}_2(r,\epsilon,\chi,d,\alpha),
\end{equation}
which gives in particular that 
\begin{equation}\label{limsupLp}
\limsup_{t \to \infty} \|u(t)\|_{L^{\frac{\chi}{\chi-r+\epsilon}}} \le \mathcal{\tilde R}_2(r,\epsilon,\chi,d,\alpha).
\end{equation}
\end{lem}
\begin{proof}For an $s>0$ (to be fixed below), we compute
\begin{align*}
\frac{1}{1+s}\frac{d}{dt}\|u\|_{L^{1+s}(\TT^d)}^{1+s}+\int_{\TT^d} u^s(x)\Lambda^{\alpha}u(x)dx\leq\left(\chi\frac{s}{1+s}-r\right)\int_{\TT^d} u^{2+s}(x)dx+r\|u\|_{L^{1+s}(\TT^d)}^{1+s},
\end{align*}
where $\Delta v = u + v \le u $ was used, see \eqref{eq:signv}.

Using Lemma \ref{lemapoincare}, we find that
\begin{align*}
\frac{1}{1+s}\frac{d}{dt}\|u\|_{L^{1+s}(\TT^d)}^{1+s}+\mathscr{P}(d,\alpha)\|u\|_{L^{1+s}(\TT^d)}^{1+s}&\leq \left(\chi\frac{s}{1+s}-r\right)\|u\|_{L^{2+s}(\TT^d)}^{2+s}+r\|u\|_{L^{1+s}(\TT^d)}^{1+s}\\
&\quad+\frac{\mathscr{P}(d,\alpha)}{(2\pi)^d}\|u\|_{L^1(\TT^d)}\left(\int_{\TT^d}u^s(x)dx\right),
\end{align*}
with $\mathscr{P}(d,\alpha)$ the constant in Lemma \ref{lemapoincare}.
We fix $\epsilon$ such that $0<\epsilon<r$. Utilizing the bounds
\begin{align*}
ry^{1+s}-\epsilon y^{2+s}&\leq r\left(\frac{r(1+s)}{\epsilon(2+s)}\right)^{1+s}-\epsilon \left(\frac{r(1+s)}{\epsilon(2+s)}\right)^{2+s}\leq r\left(\frac{r(1+s)}{\epsilon(2+s)}\right)^{1+s}\;\forall\, y\geq0,\;\epsilon>0,
\end{align*}
$$
y^s\leq y+1\;\forall\, y\geq0,\;0<s\leq 1,
$$
we have that
\begin{align*}
\frac{1}{1+s}\frac{d}{dt}\|u\|_{L^{1+s}}^{1+s}+\mathscr{P}(d,\alpha)\|u\|_{L^{1+s}}^{1+s}&\leq \left(\chi\frac{s}{1+s}-r+\epsilon\right)\|u\|_{L^{2+s}}^{2+s}+r\left(\frac{r(1+s)}{\epsilon(2+s)}\right)^{1+s}\\
&\quad+\frac{\mathscr{P}(d,\alpha)}{(2\pi)^d}\mathcal{R}_1(u_0,d)\left(\mathcal{R}_1(u_0,d)+(2\pi)^d\right).
\end{align*}
Let us define
$$
s=\frac{r-\epsilon}{\chi-r+\epsilon}. 
$$
Recall that $\chi>r$, so $s>0$ and
$$
\frac{s}{1+s}=\frac{r-\epsilon}{\chi},\;r\left(\frac{r(1+s)}{\epsilon(2+s)}\right)^{1+s}=r\left(\frac{r}{\epsilon}\frac{\chi}{2\chi-r+\epsilon}\right)^{\frac{\chi}{\chi-r+\epsilon}}. 
$$
We obtain that
\begin{align*}
\frac{1}{\frac{\chi}{\chi-r+\epsilon}}\frac{d}{dt}\|u\|_{L^{\frac{\chi}{\chi-r+\epsilon}}}^{\frac{\chi}{\chi-r+\epsilon}}+\mathscr{P}(d,\alpha)\|u\|_{L^{\frac{\chi}{\chi-r+\epsilon}}}^{\frac{\chi}{\chi-r+\epsilon}}&\leq \frac{\mathscr{P}(d,\alpha)}{(2\pi)^d}\mathcal{R}_1(u_0,d)\left(\mathcal{R}_1(u_0,d)+(2\pi)^d\right)\\
&\quad+ r\left(\frac{r}{\epsilon}\frac{\chi}{2\chi-r+\epsilon}\right)^{\frac{\chi}{\chi-r+\epsilon}}.
\end{align*}
The previous ODI can be written as
$$
\frac{d}{dt}Y(t)+\mathcal{A}Y(t)\leq \mathcal{B},
$$
for
$$
Y(t)=\|u(t)\|_{L^{\frac{\chi}{\chi-r+\epsilon}}(\TT^d)}^{\frac{\chi}{\chi-r+\epsilon}},
$$
$$
\mathcal{A}=\frac{\chi}{\chi-r+\epsilon}\mathscr{P}(d,\alpha)
$$
and
$$
\mathcal{B}=\frac{\chi}{\chi-r+\epsilon}\left( r\left(\frac{r}{\epsilon}\frac{\chi}{2\chi-r+\epsilon}\right)^{\frac{\chi}{\chi-r+\epsilon}}+\frac{\mathscr{P}(d,\alpha)}{(2\pi)^d}\mathcal{R}_1(u_0,d)\left(\mathcal{R}_1(u_0,d)+(2\pi)^d\right)\right). 
$$
Integrating in time, we find that
$$
Y(t)\leq Y(0)e^{-\mathcal{A}t}+\frac{\mathcal{B}}{\mathcal{A}}\left(1-e^{-\mathcal{A}t}\right) 
$$
hence
\[
\|u(t)\|_{L^{\frac{\chi}{\chi-r+\epsilon}}(\TT^d)} = Y^{1-\frac{r-\epsilon}{\chi}} (t) \leq Y^{1-\frac{r-\epsilon}{\chi}} (0)  e^{-\mathcal{A} ({1-\frac{r-\epsilon}{\chi}} )t}+\left(\frac{\mathcal{B}}{\mathcal{A}} \right)^{1-\frac{r-\epsilon}{\chi}}  \left(1-e^{-\mathcal{A}t}\right)^{1-\frac{r-\epsilon}{\chi}} 
\]
using that $ \left(1-e^{-\mathcal{A}t}\right)^{1-\frac{r-\epsilon}{\chi}}  \le 1-e^{-\mathcal{A}({1-\frac{r-\epsilon}{\chi}} )t}$ and the definitions of  $\mathcal{B}$ and $\mathcal{A}$ we arrive at
\[
\|u(t)\|_{L^{\frac{\chi}{\chi-r+\epsilon}}} \le \|u_0\|_{L^{\frac{\chi}{\chi-r+\epsilon}}} e^{- \mathscr{P}(d,\alpha) t } + (1- e^{- \mathscr{P}(d,\alpha) t })  \left(\frac{ r\left(\frac{r}{\epsilon}\frac{\chi}{2\chi-r+\epsilon}\right)^{\frac{\chi}{\chi-r+\epsilon}}}{\mathscr{P}(d,\alpha)}+\frac{\mathcal{R}_1(u_0,d)^2}{(2\pi)^d}+\mathcal{R}_1(u_0,d)\right)^
{1-\frac{r-\epsilon}{\chi}} 
\]
which is  \eqref{eqLp}.

Recall \eqref{LinfL3}. It implies that for any $t \ge r^{-1} \ln2$ 
$$
\|u(t)\|_{L^1}\leq 2 (2\pi)^d.
$$
We can consider our ODI not starting at the initial time, but at $t= r^{-1} \ln2$. This implies, along the lines leading to \eqref{eqLp}, the inequality \eqref{limsupLpa}.
\end{proof}
Now we can proceed with the proof of Theorem \ref{thm1} i.e. with showing the uniform estimate \eqref{boundn1}. We define $x_t$ such that
$$
\max_{y\in\TT^d} u(y,t)=u(x_t,t)=\|u(t)\|_{L^\infty(\TT^d)}.
$$
Then, due to regularity of the solution $u$, we have that $\|u(t)\|_{L^\infty(\TT^d)}$ is Lipschitz:
\[
|u(x_s,s) - u(x_t,t)| = \begin{cases} u(x_s,s) - u(x_t,t) \le u(x_s,s) - u(x_s,t) \\
 u(x_t,t) - u(x_s,s) \le  u(x_t,t) - u(x_t,s)
 \end{cases}
 \begin{rcass}
 \end{rcass}
   \sup_{y \in [x_t,x_s]^d,\tau \in [t,s]} |\partial_\tau u (\tau, y)| |s-t|.
\]
Hence due to the Rademacher theorem $\|u(t)\|_{L^\infty(\TT^d)}$ is differentiable almost everywhere. Moreover, its derivative verifies for almost every $t$
$$
\frac{d}{dt}\|u(t)\|_{L^\infty(\TT^d)} \le (\partial_t u)(x_t,t),
$$
the precise argument for the above may be found for instance in Cordoba \& Cordoba  \cite{cor2}, p.522.
In what follows we will use a few arguments based on a strict inequality for a pointwise value of $\frac{d}{dt}\|u(t)\|_{L^\infty(\TT^d)}$ at, say, $t^*$. Since it is in fact defined only almost everywhere in time, such an inequality should be understood as the inequality for $\int_{t^*}^{t^*+\delta}\frac{d}{dt}\|u(t)\|_{L^\infty(\TT^d)} dt$.
Let us fix $0<\epsilon<r$ such that
$$
\alpha>d\left(1-\frac{r-\epsilon}{\chi}\right).
$$

Then, due to Lemma \ref{keylemma}, we have that
\begin{equation}\label{eq:lem1r}
\begin{aligned}
\max_{ t \ge 0}\|u\|_{L^{\frac{\chi}{\chi-r+\epsilon}}(\TT^d)} &\leq  \|u_0\|_{L^{\frac{\chi}{\chi-r+\epsilon}}} e^{- \mathscr{P}(d,\alpha) t } + (1- e^{- \mathscr{P}(d,\alpha) t })  \mathcal{R}_2 \equiv  \mathcal{Q}_2, \\
\max_{t \ge r^{-1} \ln2} \|u(t)\|_{L^{\frac{\chi}{\chi-r+\epsilon}}} &\le \|u({t_0= r^{-1} \ln2})\|_{L^{\frac{\chi}{\chi-r+\epsilon}}} e^{- \mathscr{P}(d,\alpha) t } + (1- e^{- \mathscr{P}(d,\alpha) t })\mathcal{\tilde R}_2  \equiv  \mathcal{\tilde Q}_2\end{aligned}
\end{equation}
with $u_0$ depending $\mathcal{R}_2$ defined in \eqref{R2} and $u_0$ independent $\mathcal{R}_2$ defined in \eqref{R2t}. Let us take
$$
p=\frac{\chi}{\chi-r+\epsilon}
$$
and with this choice consider the dichotomy of Lemma  \ref{lemaaux3}. It implies that  either
\begin{equation}
u(x_t,t) \le \tag{A} \mathscr{M}_1(d,p,\alpha)\|u(t)\|_{L^p} \le \mathscr{M}_1(d,p,\alpha)\mathcal{Q}_2
\end{equation}
or
\begin{equation}
\tag{B}
\Lambda^\alpha u(x_t,t)\geq \mathscr{M}_2(d,p,\alpha)\frac{u(x_t,t)^{1+\alpha p/d}}{\|u(t)\|^{\alpha p/d}_{L^p}}.
\end{equation}
Let us introduce
\[\sigma = \frac{\alpha}{d}\frac{\chi}{\chi-r+\epsilon}-1\qquad K=\frac{\mathscr{M}_2(d,p,\alpha)}{\mathcal{Q}^{\frac{\alpha}{d}\frac{\chi}{\chi-r+\epsilon} }_2}, \qquad \tilde K=\frac{\mathscr{M}_2(d,p,\alpha)}{\mathcal{\tilde Q}^{\frac{\alpha}{d}\frac{\chi}{\chi-r+\epsilon} }_2 (r,\epsilon,\chi,d,\alpha)}  \] 
Assume now that over the evolution of $\|u(t)\|_{L^\infty(\TT^d)}$ it may take values greater than
\[
 \mathscr{M}_3 = \max\left\{ 2\|u_0\|_{L^\infty(\TT^d)} e^{-t},  2\mathscr{M}_1(d,p,\alpha)\mathcal{Q}_2, \left( \frac{4\chi}{K} \right)^{\frac{1}{2} + \frac{1}{\sigma}}, 1 \right\}
\]
Let us consider the first $t^*>0$ such that
$$
\|u(t^*)\|_{L^\infty}= \mathscr{M}_3 .
$$
Observe that $t_*>0$ thanks to the first entry of the formula for $ \mathscr{M}_3 $.  Since $\mathscr{M}_3$ exceeds (by the middle entry of its definition) the bound related to the case (A), we find ourselves at the case (B). Consequently
\[
\frac{d}{dt}\|u(t_*)\|_{L^\infty(\TT^d)} \leq (\chi-r)\|u(t_*)\|_{L^\infty(\TT^d)}^2+r\|u(t_*)\|_{L^\infty(\TT^d)}- \mathscr{M}_2(d,p,\alpha)\frac{\|u(t_*)\|_{L^\infty(\TT^d)}^{1+\alpha p/d}}{\|u(t_*)\|^{\alpha p/d}_{L^p}}.
\]
Hence via \eqref{eq:lem1r}
\begin{equation}\label{eqM2}
\begin{aligned}
\frac{d}{dt}\|u(t_*)\|_{L^\infty(\TT^d)}&\leq (\chi-r)\|u(t_*)\|_{L^\infty(\TT^d)}^2+r\|u(t_*)\|_{L^\infty(\TT^d)}-\frac{\mathscr{M}_2(d,p,\alpha)}{\mathcal{Q}^{\frac{\alpha}{d}\frac{\chi}{\chi-r+\epsilon} }_2 (r,\epsilon,\chi,d,\alpha, u_0)}\|u(t_*)\|_{L^\infty(\TT^d)}^{1+\frac{\alpha}{d}\frac{\chi}{\chi-r+\epsilon}}\\
&= (\chi-r)  \mathscr{M}_3^2 + r \mathscr{M}_3- K \mathscr{M}_3^{2+\sigma} \le \chi  \mathscr{M}_3^2 - K \mathscr{M}_3^{2+\sigma} 
\end{aligned}
\end{equation}
since $\mathscr{M}_3 \ge 1$. Observing that $\sigma \in (0, 1]$
with $ \chi  \mathscr{M}_3^2 \le \frac{K}{2}  \mathscr{M}_3^{2+\sigma} + \frac{\sigma \chi}{2 (1+ \frac{\sigma}{2})^{1 + \frac{2}{\sigma}}} \left(\frac{2 \chi}{K} \right)^\frac{2}{\sigma}$ yields
$$
\frac{d}{dt}\|u(t_*)\|_{L^\infty(\TT^d)} \le \frac{\sigma \chi}{2 (1+ \frac{\sigma}{2})^{1 + \frac{2}{\sigma}}} \left(\frac{2 \chi}{K} \right)^\frac{2}{\sigma} -  \frac{K}{2}  \mathscr{M}_3^{2+\sigma} \le  \frac{ \chi}{2} \left(\frac{2 \chi}{K} \right)^\frac{2}{\sigma} -  \frac{K}{2}  \mathscr{M}_3^{2},
$$
where we have used $\sigma \le 1$. Our choice of $\mathscr{M}_3$  (see the third entry of its definition)  gives hence
\[
\frac{d}{dt}\|u(t_*)\|_{L^\infty(\TT^d)} < 0,
\]
which is against the assumption that $t_*$ would be a first time when $\|u(t)\|_{L^\infty(\TT^d)}$ takes the value $\mathscr{M}_3$. Consequently
\[
\sup_{t \ge 0} \|u(t)\|_{L^\infty(\TT^d)} \le \mathscr{M}_3.
\]
Let us observe that, analogously to the proof of \eqref{limsupLpa} of Lemma \ref{keylemma}, we can obtain that
\[
\sup_{t \ge r^{-1} \ln 2}  \|u(t)\|_{L^\infty(\TT^d)} \le \mathscr{\tilde M}_3
\]
with
\[
 \mathscr{\tilde M}_3 = \max\left\{ 2\|u({t_0= r^{-1} \ln 2}) \|_{L^\infty(\TT^d)} e^{-t},  2\mathscr{M}_1(d,p,\alpha)\mathcal{\tilde Q}_2, \left( \frac{4\chi}{\tilde K} \right)^{\frac{1}{2} + \frac{1}{\sigma}}, 1 \right\}
\]
In order to simplify the formula for $ \mathscr{ M}_3 $ we can write
\[
 \mathscr{ M}_3 \le 2 \left( \|u_0\|_{L^\infty(\TT^d)} e^{-t}+ (\mathscr{M}_1+1) \mathcal{Q}_2+ {\left(\frac{4 \chi}{\mathscr{M}_2}\right)}^{ \frac{1}{2} +\frac{1}{\sigma}}\mathcal{Q}^{\frac{3}{\sigma}}_2 \right) \le 2 e^{-t} \|u_0\|_{L^\infty(\TT^d)} +2 \mathcal{Q}^{\frac{3}{\sigma}}_2 (\mathscr{M}_1+ {\left(\frac{4 \chi}{\mathscr{M}_2}\right)}^{ \frac{1}{2} +\frac{1}{\sigma}} +1 )
\]
and analogously 
\[
 \mathscr{ \tilde M}_3  \le 2 e^{-t}  \|u({t_0= r^{-1} \ln 2}) \|_{L^\infty(\TT^d)} +2 \mathcal{\tilde Q}^{\frac{3}{\sigma}}_2 (\mathscr{M}_1+ {\left(\frac{4 \chi}{\mathscr{M}_2}\right)}^{ \frac{1}{2} +\frac{1}{\sigma}} +1 ).
\]
These are the bounds  $\mathscr{ R}_3,  \mathscr{ \tilde R}_3$ from the thesis of Theorem \ref{thm1}, with the exception that $\mathcal{R}_2$ of  $\mathcal{Q}_2$, see \eqref{eq:lem1r}, is substituted with its upper bound $\mathcal{R}_0$, thus new $\mathcal{Q}_0$.  Theorem \ref{thm1} is proved.

\section{Proof of Theorem \ref{thm2} (stability of the homogeneous solution I) and corollary}
Let us define the new variables
$$
U=u-1,\;V=v+1.
$$
These new variables solve
\begin{align}\label{eqU}
\pat U&=- \Lambda^\alpha U+\chi\nabla\cdot((U+1) \nabla V)-r(U+1)U,\text{ in }(x,t)\in \TT^d\times[0,\infty)\\
\Delta V -V&=U,\text{ in }(x,t)\in \TT^d\times[0,\infty)\label{eqV}\\
U(x,0)&=u_0(x)-1\text{ in }x\in \TT^d.\label{eqU2}
\end{align}
Furthermore,  \eqref{LinfL1}  and \eqref{LinfL12} imply
$$
\|U(t)\|_{L^1}\leq (2\pi)^d+\max\left\{(\|u_0\|_{L^1(\TT^d)},(2\pi)^d\right\}.
$$
Let us define
$$
\overline{U}(t)=U(\overline{x}_t,t)=\max_x U(x,t),
$$
$$
\underline{U}(t)=U(\underline{x}_t,t)=\min_x U(x,t).
$$
Due to \eqref{boundn1} 
$$
-1\leq U(x,t)\leq (1+\delta)  \mathscr{ R}_3 (t; r,\epsilon,\chi,d,\alpha, u_0)  -1,
$$
$$
-\overline{U}(t)\leq V(x,t)\leq -\underline{U}(t).
$$
Using the pointwise method already described in the proof of Theorem \ref{thm1}, we have consequently that
\begin{align*}
\frac{d}{dt}\overline{U}(t)&= -\Lambda^{\alpha}U(\overline{x}_t,t)+\chi(\overline{U}(t)+1)(\overline{U}(t)+V(\overline{x}_t,t))-r(\overline{U}(t)+1)\overline{U}(t)\\
&\leq-\Lambda^{\alpha}U(\overline{x}_t,t)+(\chi-r)\overline{U}(t)^2 +(\chi-r)\overline{U}(t)-\chi\overline{U}(t)\underline{U}(t)-\chi\underline{U}(t),
\end{align*}
\begin{align*}
\frac{d}{dt}\underline{U}(t)&= -\Lambda^{\alpha}U(\underline{x}_t,t)+\chi(\underline{U}(t)+1)(\underline{U}(t)+V(\underline{x}_t,t))-r(\underline{U}(t)+1)\underline{U}(t)\\
&\geq-\Lambda^{\alpha}U(\underline{x}_t,t)+(\chi-r)\underline{U}(t)^2 +(\chi-r)\underline{U}(t)-\chi\overline{U}(t)\underline{U}(t)-\chi\overline{U}(t).
\end{align*}
Collecting both estimates, we obtain
\begin{align*}
\frac{d}{dt}\left(\overline{U}(t)-\underline{U}(t)\right)&\leq -\Lambda^{\alpha}U(\overline{x}_t,t)+\Lambda^{\alpha}U(\underline{x}_t,t)
+(\chi-r)\left(\overline{U}(t)^2 -\underline{U}(t)^2\right)\\
&\quad+(\chi-r)\left(\overline{U}(t)-\underline{U}(t)\right)+\chi\left(\overline{U}(t)-\underline{U}(t)\right)\\
&\leq -\Lambda^{\alpha}U(\overline{x}_t,t)+\Lambda^{\alpha}U(\underline{x}_t,t)
+\left(\overline{U}(t)-\underline{U}(t)\right)\left[2\chi-r+(\chi-r)\left(\overline{U}(t)+\underline{U}(t)\right)\right].
\end{align*}
Now let us note that
\begin{align*}
\Lambda^{\alpha}U(\overline{x}_t,t)&\geq\mathscr{C}_{d,\alpha}\text{P.V.}\int_{\TT^d}\frac{u(\overline{x}_t,t)-u(\overline{x}_t-y,t)dy}{|y|^{d+\alpha}}\\
&\geq\mathscr{C}_{d,\alpha}\text{P.V.}\int_{\TT^d}\frac{u(\overline{x}_t,t)-u(\overline{x}_t-y,t)dy}{({2}\pi\sqrt{d})^{d+\alpha}}\\
&\geq \mathscr{C}_{d,\alpha}\frac{(2\pi)^d\overline{U}(t)-\int_{\TT^d}U(y,t)dy}{({2}\pi\sqrt{d})^{d+\alpha}}.
\end{align*}
Similarly
$$
-\Lambda^{\alpha}U(\underline{x}_t,t)\geq\mathscr{C}_{d,\alpha}\text{P.V.}\int_{\TT^d}\frac{-u(\underline{x}_t,t)+u(\underline{x}_t-y,t)dy}{|y|^{d+\alpha}}\geq \mathscr{C}_{d,\alpha}\frac{-(2\pi)^d\underline{U}(t)+\int_{\TT^d}U(y,t)dy}{({2}\pi\sqrt{d})^{d+\alpha}}.
$$
Thus
\begin{align*}
\frac{d}{dt}\left(\overline{U}(t)-\underline{U}(t)\right)&\leq -\mathscr{C}_{d,\alpha}\frac{(2\pi)^d\left(\overline{U}(t)-\underline{U}(t)\right)}{(\pi\sqrt{d})^{d+\alpha}}+\left(\overline{U}(t)-\underline{U}(t)\right)\left[2\chi-r+(\chi-r)\left(\overline{U}(t)+\underline{U}(t)\right)\right]\\
&\leq \left(\overline{U}(t)-\underline{U}(t)\right)\left[2\chi-r+(\chi-r)2\overline{U}(t)-\frac{(2\pi)^d\mathscr{C}_{d,\alpha}}{({2}\pi\sqrt{d})^{d+\alpha}}\right]-(\chi-r)\left(\overline{U}(t)-\underline{U}(t)\right)^2.
\end{align*}
Therefore, if
$$
-\gamma=2\chi-r+(\chi-r)2\left(\mathscr{ R}_3 (t; r,\epsilon,\chi,d,\alpha, u_0)  -1\right)-\frac{(2\pi)^d\mathscr{C}_{d,\alpha}}{({2}\pi\sqrt{d})^{d+\alpha}} <0,
$$
then there exists such $\delta$ that
$$
2\chi-r+(\chi-r)2\left( (1+\delta) \mathscr{ R}_3 (t; r,\epsilon,\chi,d,\alpha, u_0) -1\right)-\frac{(2\pi)^d\mathscr{C}_{d,\alpha}}{({2}\pi\sqrt{d})^{d+\alpha}} \le 0.
$$
Hence 
$$
\left(\overline{U}(t)-\underline{U}(t)\right)\leq \left(\overline{U}(0)-\underline{U}(0)\right)e^{-\gamma t}\rightarrow 0.
$$
Translating the previous inequality into our original variable $u$, we obtain that
$$
\left(\|u(t)\|_{L^\infty(\TT^d)}-\min_{x\in\TT^d}u(x,t)\right)\leq \left(\|u_0\|_{L^\infty(\TT^d)}-\min_{x\in\TT^d}u_0(x)\right)e^{-\gamma t}\rightarrow 0.
$$
This inequality implies that the solution $u$ converges to a constant $c_u$. Similarly, $v$ converges to another constant $-c_u$. There are only two possible steady state solutions that are constants, namely $(1,-1)$ and $(0,0)$. However, it is easy to see that the case $(0,0)$ is unstable. Indeed, our nonlocal fractional diffusion manifest its another useful feature here. Namely, for the classical laplacian one usually discards the possibility of vanishing of a solution (thus of staying in the case $(0,0)$) via assuming that $\min u_0 >0$, since then $\min_x u (t,x) >0$. In our case it suffices to have an initial data $u_0$ not identically zero. Indeed, assume that $0=\min_y u(y,t)$. Then, if we write $\underline{x}_t$ for the point such that $\min_y u(y,t)=u(\underline{x}_t,t)$, we have that for 
$u(t)$ not identically zero
$$
\partial_t u(\underline{x}_t,t)=-\Lambda^\alpha u(\underline{x}_t,t)+\chi u(\underline{x}_t,t)(u(\underline{x}_t,t)+v(\underline{x}_t,t))+ru(\underline{x}_t,t)(1-u(\underline{x}_t,t))=-\Lambda^\alpha u(\underline{x}_t,t)>0
$$ thanks to 
$$
-\Lambda^\alpha u(\underline{x}_t,t)>0 \; \text{ for $u(t)$ not identically zero}. 
$$
The only left scenario is then such that solutions vanishes uniformly. But this due to time-continuity demands that both $u$ and $v$ are close to $0$ uniformly earlier, and then ODI homeostatic part prevents further approaching zero (observed by looking again at the minimum of such a nonzero and close to zero solution).

We can write hence
\begin{align*}
\|u(t)-1\|_{L^\infty(\TT^d)}&=\max\{\|u(t)\|_{L^\infty(\TT^d)}-1,1-\min_{x\in\TT^d}u(x,t)\}\\
&\leq\|u(t)\|_{L^\infty(\TT^d)}-1+1-\min_{x\in\TT^d}u(x,t)\\
&\leq \left(\|u_0\|_{L^\infty(\TT^d)}-\min_{x\in\TT^d}u_0(x)\right)e^{-\gamma t}.
\end{align*}
Theorem \ref{thm2} is therefore shown.

In order to prove Corollary \ref{cor:1}, it suffices to observe that in the above proof one may replace every $\mathscr{ R}_3 (t; r,\epsilon,\chi,d,\alpha, u_0) $ with the data independent $\mathcal{R}_\infty(r,\epsilon,\chi,\alpha,d)$, at the cost of considering only sufficiently large times, see \eqref{boundninf}. 

\section{Proof of Theorem \ref{thm2b} (stability of the homogeneous solution II)}
We consider now the case $d=\alpha=1$. Let us pick a small parameter  $\min\{ 1,(4 \pi \chi)^{-1} -2 \pi \} \ge \delta>0$.  As a consequence of \eqref{LinfL12} we know that there exists a transient time $t^*(u_0,r,\delta)$ such that
\begin{equation}\label{t3a}
\|u(t)\|_{L^1}\leq 2\pi+\delta\quad\forall\,t\geq t^*.
\end{equation}

We will restrict our analysis to $t\geq t^*$.

Let us we denote by $\overline{x}_t$, $\underline{x}_t$ the points such that
$$
\overline{u}(t)=\max_y u(y)=u(\overline{x}_t),\;\underline{u}(t)=\min_y u(y)=u(\underline{x}_t)
$$
Note that we have
$$
0\leq \underline{u}(t)\leq 1+\frac{\delta}{2\pi}.
$$
As in the proof of Theorem \ref{thm2}, we obtain
$$
\frac{d}{dt}\overline{u}\leq-\Lambda u(\overline{x}_t)+\chi\overline{u}(\overline{u}+v(\overline{x}_t,t))+r\overline{u}(1-\overline{u}),
$$
$$
\frac{d}{dt}\underline{u}\geq-\Lambda u(\underline{x}_t)+\chi\underline{u}(\underline{u}+v(\underline{x}_t,t))+r\underline{u}(1-\underline{u}).
$$
Using \eqref{propv}, we have that
$$
-\underline{u}\geq v(\overline{x}_t,t),\;v(\underline{x}_t,t)\geq-\overline{u}.
$$
Via \eqref{eq:9b} in Appendix A we also compute
$$
\Lambda u(\overline{x}_t)\geq \frac{1}{4\pi}\left(2\pi \overline{u}-\int_\TT u(y,t)dy\right),\;-\Lambda u(\underline{x}_t)\geq \frac{-1}{4\pi}\left(2\pi \underline{u}-\int_\TT u(y,t)dy\right)
$$
so
$$
\frac{d}{dt}\overline{u}\leq\frac{-1}{4\pi}\left(2\pi \overline{u}-\int_\TT u(y,t)dy\right)+\chi\overline{u}(\overline{u}-\underline{u})+r\overline{u}(1-\overline{u}),
$$
$$
\frac{d}{dt}\underline{u}\geq\frac{-1}{4\pi}\left(2\pi \underline{u}-\int_\TT u(y,t)dy\right)+\chi\underline{u}(\underline{u}-\overline{u})+r\underline{u}(1-\underline{u}),
$$
hence together
\begin{equation}\label{eq:s3f}
\frac{d}{dt}\left(\overline{u} -\underline{u} \right) \leq -\frac{\overline{u}-\underline{u}}{2}+(\chi-r)(\overline{u}-\underline{u})(\overline{u}+\underline{u})+r(\overline{u}-\underline{u}).
\end{equation}
Lemma \ref{lemaaux3} now says that either
\begin{enumerate}
\item[(i)] $
(4 + \frac{2 \delta}{\pi} \ge)\; \frac{2}{\pi}\|u(t)\|_{L^1} \geq \|u(t)\|_{L^\infty} \; (\geq \frac{\|u(t)\|_{L^1}}{2\pi}),
$
or
\item[(ii)] $
\Lambda u(\overline{x}_t,t)\geq \frac{1}{4\pi}\frac{u(\overline{x}_t)^{2}}{\|u(t)\|_{L^1}} \;(\geq \frac{1}{4\pi}\frac{u(\overline{x}_t)^{2}}{2\pi+\delta}).
$
\end{enumerate} 
Let us now argue that there exists  a time $t^{**} \ge t^*$ such that $\overline{u} (t^{**} ) \le 4+\frac{2\delta}{\pi}$. Assume otherwise, i.e. for all $t \ge t^*$ it holds $\overline{u} (t) > 4+\frac{2\delta}{\pi}$. This excludes the case (i) and hence as in the proof of Theorem \ref{thm1} we have that
\begin{equation}\label{eq:s3g}
\frac{d}{dt}\|u(t)\|_{L^\infty} \leq \left(\chi-\frac{1}{4\pi (2\pi+\delta)}\right)\|u(t)\|_{L^\infty}^2+r\|u(t)\|_{L^\infty}(1-\|u(t)\|_{L^\infty}) \le r\|u(t)\|_{L^\infty}(1-\|u(t)\|_{L^\infty}),
\end{equation}
where the second inequality comes from the assumed 
$$
\chi-(8\pi^2)^{-1}<0
$$
and our choice of $\delta \le (4 \pi \chi)^{-1} -2 \pi$. Consequently $\|u(t)\|_{L^\infty}$ approaches $1$, therefore the assumption  $\overline{u} (t) > 4+\frac{2\delta}{\pi}$ for all $t \ge t^*$ is false. There must be then a finite 
$t^{**}$ at which $\overline{u} (t^{**}) \le 4+\frac{2\delta}{\pi}$. 

If for all $t  > t^{**}$ it still holds $\overline{u} (t) \le 4+\frac{2\delta}{\pi}$, then $(\overline{u}+\underline{u}) \le 8+\frac{4\delta}{\pi}$ and thanks to \eqref{eq:s3f}
\[
\frac{d}{dt}\left(\overline{u}-\underline{u}\right) \leq \left[(\chi-r)(8+\frac{4\delta}{\pi})+r-\frac{1}{2}\right](\overline{u}-\underline{u}) = \epsilon (\overline{u}-\underline{u})
\]
with $\epsilon \le 8 \chi - \frac{1}{2} + \frac{4\delta}{\pi} \chi < \frac{1}{\pi^2}- \frac{1}{2} + \frac{\delta}{2 \pi^3} \le \frac{1}{2} (\frac{1}{\pi} -1)$ , where we have used the before needed
$\chi-\frac{1}{8\pi^2}<0$ and $\delta\le 1$.
Hence
$$
\left(\overline{u}(t)-\underline{u}(t)\right) \le (\overline{u}(t^{**})-\underline{u}(t^{**}) ) e^{\frac{1}{2} (\frac{1}{\pi} -1) t} \le  (8+\frac{4\delta}{\pi}) e^{\frac{1}{2} (\frac{1}{\pi} -1) t} \le 10 e^{\frac{1}{2} (\frac{1}{\pi} -1) t}.
$$
As before, this implies that the solution $u$ converges to a constant. This constant can only be $0$ or $1$ and we have seen that $0$ is unstable.

Finally it could happen that  $\overline{u} (t^{**}) \le 4+\frac{2\delta}{\pi}$, but at some further time  $\overline{u} (t) > 4+\frac{2\delta}{\pi}$, Hence there exists $t_\dagger$ such that $\overline{u} (t_\dagger) = 4+\frac{2\delta}{\pi}$ and immediately before $t_\dagger$ we have $\overline{u} ({t_\dagger}^-) < 4+\frac{2\delta}{\pi}$: this is merely continuity in the case $\overline{u} (t^{**}) < 4+\frac{2\delta}{\pi}$ and observation that $
\frac{d}{dt}\|u(t^{**})\|_{L^\infty} <0 $ provided $\overline{u} (t^{**}) = 4+\frac{2\delta}{\pi}$, compare \eqref{eq:s3g}, so the  in this case $\overline{u}$ must drop below $4+\frac{2\delta}{\pi}$ immediately after $t^{**}$.
But now the existence of $t_\dagger$ is contradicted again by  \eqref{eq:s3g} giving $\frac{d}{dt}\|u(t_\dagger)\|_{L^\infty} <0 $ i.e. $\overline{u} ({t_\dagger}^-) > 4+\frac{2\delta}{\pi}$.

\section{Proof of Theorem \ref{thm3}. Uniqueness}

For a number $p$, let us denote by $p^+$ any number larger than $p$ and by $p^-$ any number smaller than $p$. In particular, $\infty^-$ is any finite number.

Let us take two distributional solutions $u_1, u_2$ to \eqref{eqDD} starting from the same initial datum $u_0$ and belonging to $L^2 (L^{2^+}))$.  Consequently for $\bu =u_1 -  u_2$ one has
\begin{equation}\label{eq:ud}
\pat \bu=- \Lambda^\alpha \bu+\chi \nabla\cdot(\bu B (u_1) ) +\chi \nabla\cdot(u_2  B (\bu) ) + r \bu - r \bu (u_1 + u_2)
\end{equation}
in the sense of distribution, where 
$$
B(u)=\nabla(\Delta-1)^{-1}u.
$$ 
Let us multiply \eqref{eq:ud} with a sufficiently regular  $\psi$. Hence
\begin{equation}\label{eq:psi1}
\begin{aligned}
&\int \pat \bu \psi = \\
&- \int \Lambda^{\alpha + \rho -1} \bu \Lambda^{1-\rho} \psi + \chi \int \Lambda^{\rho} R \cdot (v   B (u_1)) \Lambda^{1-\rho} \psi + \chi \int \Lambda^\rho  R \cdot (u_2   B (\bu)) \Lambda^{1-\rho}  \psi  \\
&+  r  \int \bu \psi - r  \int \bu  (u_1 + u_2) \psi,
\end{aligned}
\end{equation}
where integrals are over space and time.
Observe that for a given $\rho \in [0,1)$ and $\psi \in L^2 (H^{1-\rho} )$ the first term on the r.h.s. above is finite provided $u_1, u_2$ belongs to $L^2 (H^{\alpha+\rho-1} )$. Let us consider conditions for  finiteness of further terms on the r.h.s. for $\psi \in L^2 (H^{1-\rho} )$. Concerning the integrals involving $\chi$, when the differentiation does not hit $B$, it suffices that $ L^2 (H^{\rho^+})$ (or  $ L^2 (H^{\rho})$ for $d=1$. Let us continue with the case $d=2$ only and observe in the analogous computations for the case $d=1$ one does not need $\cdot^+$), since for $d \le2$ 
\begin{equation}\label{eq:buq}
\|B (u)\|_{L^\infty (L^{\infty^-})}\le {c\|B (u)\|_{L^\infty (\dot{H}^{1})}}\le C\|u\|_{L^\infty (L^2)}.
\end{equation}
When the differentiation hits $B$, since
\[
 \| \bu (t)\|_{L^{p_3}}   \| \Lambda^{\rho}  B (u_1 (t)) \|_{L^{p'_3}}  \|\Lambda^{1-\rho} \psi (t) \|_{L^2} \le C \| u_2 (t)\|_{H^{\rho}}   \|v (t) \|_{L^2} \|\psi (t)\|_{H^{1-\rho}},
\]
provided 
\[
\rho -1 - \frac{d}{p'_3} \le -\frac{d}{2}, \quad - \frac{d}{p_3} \le \rho - \frac{d}{2},
\]
we need $u_1, u_2 \in L^\infty (L^2) \cap L^2 (H^\rho)$. The same holds for the other integral involving $\chi$. Finally, to deal with quadratic terms involving $r$, we observe that by embedding
\[
\int |\bu  (u_1 + u_2) \psi| \le C (\|u^2_1\|_{L^2 (L^\xi)} +  \|u^2_2\|_{L^2 (L^\xi)} ) \|\psi\|_{L^2(H^{1-\rho})},
\]
for $\xi = \frac{2d}{2 - 2 \rho +d}$ and for finiteness of $L^2 (L^\xi)$ norms we interpolate $L^\infty (L^2)$ and $L^2 (H^{(\rho -1 + d/2)^+})$. So  $u_1, u_2 \in  L^2 (H^{\rho^+}) \cap L^\infty (L^2)$ suffices here. Putting together all our requirements, we see that  for $\psi \in L^2 (H^{1-\rho} )$ in \eqref{eq:psi1} it is enough to have
\begin{equation}\label{con3}
u_1, u_2 \in L^\infty (L^2) \cap L^2 (H^{\rho^+}) \cap L^2 (H^{1-\rho}) \cap L^2 (H^{\alpha+\rho-1} )
\end{equation}
Then l.h.s. of \eqref{eq:psi1} is finite, i.e. $\pat \bu \in L^2 (H^{\rho-1})$ and since $u_1, u_2 \in L^2 (H^{1- \rho})$, then by interpolation $\bu \in C(L^2)$. 

Next, in view of the assumed $u_1, u_2 \in L^2 (H^\alpha)$, we can have $\psi = \bu$, choosing $1-\rho = \alpha/2$. Consequently
\[
\frac{1}{2} \frac{d}{dt}\|\bu(t)\|^2_{L^2} + \|\bu(t)\|^2_{ \dot{H}^\frac{\alpha}{2}}  \leq \chi \left|  \langle \bu B (u_1) \nabla v \rangle \right|+  \chi \left| \int \Lambda^{1-\alpha/2} R \cdot (v   B (u_1)) \Lambda^{\alpha/2} \bu \right|   +  r   \|\bu(t)\|^2_{L^2},
\]
where we have also used nonnegativity and $\langle \cdot \rangle$  denotes  the duality pairing $\langle \cdot \rangle_{H^\rho, H^{1- \rho}}$. 
Hence, using very weak integration by parts 
\[
\frac{1}{2} \frac{d}{dt}\|\bu(t)\|^2_{L^2} + \|\bu(t)\|^2_{\dot{H}^\frac{\alpha}{2}}  \leq \frac{\chi}{2} \int \nabla \cdot B (u_1) (t)  \bu^2 (t)  + \chi \int |\Lambda^{1-\alpha/2} R \cdot (u_2   B (\bu))  | |\Lambda^{\alpha/2} \bu| + r   \|\bu(t)\|^2_{L^2} 
\]
and consequently, since $\|\nabla \cdot B f \|_p \le C\| f \|_p$
\begin{multline}\label{eq:ud2}
\frac{1}{2} \frac{d}{dt}\|\bu(t)\|^2_{L^2} + \|\bu(t)\|^2_{H^\frac{\alpha}{2}}  \leq \\
\frac{\chi}{2}  \| u_1 (t) \|_{L^{p_1}} \| \bu (t) \|^2_{L^{2p'_1}} + C  \|\Lambda^{1-\alpha/2} u_2 (t)\|_{L^{2^+}}   \|B (\bu (t)) \|_{L^{\infty^-}}  \|\Lambda^{\alpha/2} \bu (t) \|_{L^{2}} + \\
C  \| u_2 (t)\|_{L^{p_3}}   \| \Lambda^{1-\alpha/2}  B (\bu (t)) \|_{L^{p'_3}}  \|\Lambda^{\alpha/2} \bu (t) \|_{L^2}  +  r   \|\bu(t)\|^2_{L^2}  \\
=: \frac{\chi}{2} I  + C II + C III +  r   \|\bu(t)\|^2_{L^2}.
\end{multline}
Using interpolation and embeddings, we estimate the terms on r.h.s. as follows (suppressing $t$ for a moment)
\[
\begin{aligned}
& I \le \| u_1 \|_{L^{p_1}} \| \bu \|^{2\theta}_{L^2}  \| \bu \|^{2(1-\theta)}_{H^\frac{\alpha}{2}}  \le \epsilon  \| \bu \|^2_{H^\frac{\alpha}{2}}  +C_\epsilon \| u_1 \|^\frac{1}{\theta} _{L^{p_1}} \|\bu \|^{2}_{L^2},  \qquad \text{ where }  \frac{1}{p_1'} = \frac{\alpha}{d} (\theta -1) + 1\\
& II \le   C\|\Lambda^{1-\alpha/2} u_2\|_{L^{2^+}}  \|v \|_{L^2}   \|\Lambda^{\alpha/2} \bu \|_{L^2} \le C\|u_2\|_{H^{(1-\alpha/2)^+} } \|v \|_{L^2}   \| \bu \|_{H^{\alpha/2}} \le C\|u_2\|_{H^{\alpha} } \|v \|_{L^2}   \| \bu \|_{H^{\alpha/2}}\\
& III \le C_\epsilon\| u_2\|^2_{L^{p_3}}  \|v \|^2_{L^2}  + \epsilon \|\bu\|^2_{H^\frac{\alpha}{2}} \qquad \text{ for }   -\alpha/2 -\frac{d}{p'_3} \le - \frac{d}{2}.
\end{aligned}
\]
\[
\frac{\chi}{2}  \| u_1 (t) \|_{L^{p_1}} \| \bu (t) \|^2_{L^{2p'_1}} 
\]
More precisely, the middle inequality involves \eqref{eq:buq} and $\alpha > (1-\alpha/2)^+$ by assumed $\alpha>1$. The last inequality uses the fact that $\Lambda^{1-\alpha/2}  B$ is on the Fourier side $\sim |\xi|^{1- \alpha/2} \frac{\xi}{|\xi|^2 +1}$, hence there is no problem with null modes for  $W^{-\alpha/2, p'_3}$ so by embedding for $  -\alpha/2 -\frac{d}{p'_3} \le - \frac{d}{2}$, $p'_3 \in (1, \infty)$
\[
 \| \Lambda^{1-\alpha/2}  B (\bu ) \|_{L^{p'_3}} \le   C  \| \bu \|_{L^2}.
\]
Consequently
\[
III \le   C_\epsilon \|u_2\|^2_{H^{(1-\alpha/2)} }  \|v \|^2_{L^2}  + \epsilon \|\bu\|^2_{H^\frac{\alpha}{2}}  \le  C_\epsilon \|u_2\|^2_{H^{\alpha} }  \|v \|^2_{L^2}  + \epsilon \|\bu\|^2_{H^\frac{\alpha}{2}}  \quad \text{ provided } p_3 = \frac{2d}{d- 2 (1-\alpha/2)}
\]
This last choice of $p_3$ is within the needed before condition  $-\alpha/2 -\frac{d}{p'_3} \le - \frac{d}{2}$. Altogether, the above estimates yield in \eqref{eq:ud2}
\[
\frac{d}{dt}\|\bu(t)\|^2_{L^2} + \|\bu(t)\|^2_{H^\frac{\alpha}{2}}  \leq C \| u_1 (t) \|^\frac{1}{\theta} _{L^{p_1}} \|\bu (t) \|^{2}_{L^2}  + C \|u_2 (t)\|^2_{H^{(1-\alpha/2)^+} }   \| \bu(t) \|^2_{L^2} + C   \|\bu(t)\|^2_{L^2}.
\]
Hence for uniqueness we need only to check whether 
\[
\int_0^T  \| u_1 (t) \|^\frac{1}{\theta}_{L^{p_1}}  dt < \infty,
\]
where  $\frac{1}{\theta} = \frac{\alpha p_1}{\alpha p_1 -d}$. Requiring $ \frac{1}{\theta} = 2$, we need $p_1 = 2d/ \alpha$, so $u_1 \in  L^2 (H^{\frac{d- \alpha}{2}}) $ is sufficient. 

\appendix

\section{On the fractional Laplacian}
The $d-$dimensional fractional (minus) Laplacian $\Lambda^{\alpha}$ is defined through the Fourier transform as
$$
\widehat{\Lambda^\alpha u}(\xi)=|\xi|^\alpha \hat{u}(\xi).
$$
This operator also enjoys the following representation as a singular integral (see \cite{Ghyperparweak} for an elementary derivation):
\begin{equation}\label{eq:1b.1.5}
\Lambda^{\alpha}u=\mathscr{C}_{d,\alpha} \text{P.V.}\int_{\RR^d}\frac{u(x)-u(y)}{|x-y|^{d+\alpha}}dy,
\end{equation}
where
$$
\mathscr{C}_{d,\alpha}=2\left(\int_{\RR^d}\frac{4\sin^2\left(\frac{x_1}{2}\right)}{|x|^{d+\alpha}} dx\right)^{-1}.
$$

In the case of periodic functions, we have the following equivalent representation
\begin{align}\label{eq:9}
\Lambda^\alpha u(x)&=\mathscr{C}_{d,\alpha}\bigg{(}\sum_{k\in \ZZ^d, k\neq 0}\int_{\TT^d}\frac{u(x)-u(x-y)dy}{|y+2k\pi|^{d+\alpha}}+\text{P.V.}\int_{\TT^d}\frac{u(x)-u(x-y)dy}{|y|^{d+\alpha}}\bigg{)}.
\end{align}

Let us emphasize that in the case $d=1=\alpha$, the previous series can be computed and results in

\begin{align}\label{eq:9b}
\Lambda u(x)&=\frac{1}{4\pi}\int_{\TT}\frac{u(x)-u(x-y)dy}{\sin^2\left(y/2\right)}.
\end{align}
Then we have the following results

\begin{lem}[\cite{BGK, BG4, burczak2016suppression, Ghyperparweak}]\label{lemaentropy2}
Let $0<s$,  \;$0<\alpha<2$, \; $0<\delta<\alpha/(2+2s)$ and $d \ge 1$. Then for a sufficiently smooth $u \ge 0$ it holds
\begin{equation}\label{ene:1}
\frac{4s}{(1+s)^2}  \int_\Td  | \Lambda^\frac{\alpha}{2} (u^\frac{s+1}{2}) |^2dx \le \int_{\Td}\Lambda^\alpha u(x) u^s(x)dx.
\end{equation}
If additionally $s \le 1$, then
\begin{equation}\label{ene:2}
\|u\|_{\dot{W}^{\alpha/(2+2s)-\delta,1+s} (\Td)}^{2+2s}\leq \mathscr{S}(\alpha,s,\delta, d)\|u\|_{L^{1+s} (\Td) }^{1+s} \int_{\Td}\Lambda^\alpha u(x) u^s(x)dx,
\end{equation}
where $\mathscr{S}(\alpha,s,\delta, d)$ can be taken as
$$
\mathscr{S}(\alpha,s,\delta, d)=\frac{2^{2s+1}}{\mathscr{C}_{d,\alpha}s}\sup_{x\in\TT^d}\int_{\TT^d}\frac{1}{|x-y|^{d-2(1+s)\delta}}dy
$$

Furthermore, for a sufficiently smooth $u \ge 0$, $0<\alpha<2$, $0<\delta<\alpha/2$ and $d \ge 1$, the extremal case $s=0$ holds
\begin{equation}\label{ene:3}
\|u\|_{\dot{W}^{\alpha/2-\delta,1}(\TT^d)}^2\leq \mathscr{S}(\alpha,0,\delta, d)\|u\|_{L^1(\TT^d)}\int_{\Td}\Lambda^\alpha u(x)\log(u(x))dx,
\end{equation}
with
$$
\mathscr{S}(\alpha,0,\delta, d)=\frac{2}{\mathscr{C}_{d,\alpha}}\sup_{x\in\TT^d}\int_{\TT^d}\frac{1}{|x-y|^{d-2\delta}}dy.
$$
\end{lem}

\begin{lem}\label{lemapoincare}
Let $0\leq u\in L^{1+s}(\TT^d)$, $0<s<\infty$, be a given function and $0<\alpha<2$, be a fixed constant. Then, 
$$
\mathscr{P}(d,\alpha)\|u\|_{L^{1+s}}^{1+s}\le \int_{\TT^d}\Lambda^\alpha u(x) u^s(x)dx+\frac{\mathscr{P}(d,\alpha)}{\left(2\pi\right)^d}\left(\int_{\TT^d}u(x)dx\right)\left(\int_{\TT^d}u^s(x)dx\right),
$$
for 
$$
\mathscr{P}(d,\alpha)= \frac{4\left(\int_{\RR^d}\frac{4\sin^2\left(\frac{x_1}{2}\right)}{|x|^{d+\alpha}} dx\right)^{-1}}{\left(2\pi \right)^{\alpha}d^{\frac{d+\alpha}{2}}}.
$$
Furthermore, in the case $d=\alpha=1$, this constant can be taken
$$
\mathscr{P}(1,1)=1.
$$
\end{lem}
\begin{proof}Note that 
$$
\sup_{x,y\in\TT^d}|x-y|=\text{length of $d$-dimensional hypercube's longest diagonal}= 2\pi \sqrt{d}
$$
Due to the positivity of the terms
$$
0\leq \mathscr{C}_{d,\alpha}\sum_{k\in \ZZ^d, k\neq 0}\int_{\TT^d} u^s(x)\int_{\TT^d}\frac{u(x)-u(x-y)dy}{|y+2k\pi|^{d+\alpha}}dx,
$$
we have that
\begin{align*}
\int_{\TT^d}u^s(x)\Lambda^{\alpha}u(x)dx&\geq \mathscr{C}_{d,\alpha}\int_{\TT^d}\text{P.V.}\int_{\TT^d}\frac{(u(x)-u(y))(u^s(x)-u^s(y))}{|x-y|^{d+\alpha}}dy dx\\
&\geq \frac{\mathscr{C}_{d,\alpha}}{\left(2\pi \sqrt{d}\right)^{d+\alpha}}\int_{\TT^d}\int_{\TT^d}(u(x)-u(y))(u^s(x)-u^s(y))dy dx\\
&\geq \frac{2\mathscr{C}_{d,\alpha}(2\pi)^d}{\left(2\pi \sqrt{d}\right)^{d+\alpha}}\|u\|_{L^{1+s}}^{1+s}-\frac{2\mathscr{C}_{d,\alpha}}{\left(2\pi \sqrt{d}\right)^{d+\alpha}}\left(\int_{\TT^d}u(x)dx\right)\left(\int_{\TT^d}u^s(x)dx\right).
\end{align*}
Then, we obtain that
$$
\mathscr{P}(d,\alpha)= \frac{2\mathscr{C}_{d,\alpha}}{\left(2\pi \right)^{\alpha}d^{\frac{d+\alpha}{2}}}.
$$
In the case $d=\alpha=1$, we have that
$$
\Lambda u(x)=\frac{1}{4\pi}\text{P.V.}\int_\TT \frac{u(x)-u(x-y)}{\sin^2(y/2)}dy.
$$
Thus, repeating the argument using $\sin^2\leq 1$, we find that
$$
\mathscr{P}(1,1)=1.
$$
\end{proof}
Quite remarkably, the nonlocal character of the fractional Laplacian allows for pointwise estimates:
\begin{lem}[\cite{Gsemiconductor, BG4, AGM}]\label{lemaaux3}
Let $h\in C^2(\TT^d)$ be a function. Assume that $h(x^*):=\max_x h(x)>0$. Then, there exists two constants $\mathscr{M}_i(d,p,\alpha),$ $i=1,2$ such that either
$$
\mathscr{M}_1(d,p,\alpha)\|h\|_{L^p} \geq h(x^*),
$$
or
$$
\Lambda^\alpha h(x^*)\geq \mathscr{M}_2(d,p,\alpha)\frac{h(x^*)^{1+\alpha p/d}}{\|h\|^{\alpha p/d}_{L^p}},
$$
with
$$
\mathscr{M}_1(d,p,\alpha)=\left(\frac{\pi^{d/2}}{2^{1+p}}\int_{0}^\infty z^{d/2}e^{-z}dz\right)^{1/p}
$$
and
$$
\mathscr{M}_2(d,p,\alpha)=\mathscr{C}_{d,\alpha}\frac{\left(\frac{\pi^{d/2}}{\int_{0}^\infty z^{d/2}e^{-z}dz}\right)^{1+\alpha/d}}{4\cdot 2^{\frac{(p+1)\alpha}{d}}}.
$$
Furthermore, in the case $d=1=\alpha$, we have that $\mathscr{M}_i(1,1,1)$ can be taken as
$$
\mathscr{M}_1(1,1,1)=\frac{2}{\pi},\;\mathscr{M}_2(1,1,1)=\frac{1}{4\pi}.
$$

\end{lem}

\section*{Acknowledgements}
RGB was supported by the LABEX MILYON (ANR-10-LABX-0070) of Universit\'e de Lyon, within the program ``Investissements d'Avenir'' (ANR-11-IDEX-0007) operated by the French National Research Agency (ANR).

\bibliographystyle{abbrv}
\bibliography{bibliografia}

\end{document}